\documentclass[12pt]{amsart}

\title[Central limit theorem for components in meandric systems]
{Central limit theorem for components in meandric systems through
  high moments}
\date{3 March, 2023} 

\usepackage{a4wide}
\usepackage{times}
\usepackage{bbm}
\usepackage{mathtools, amssymb}
\usepackage{graphicx,xspace}
\usepackage{epsfig}
\usepackage{dsfont}
\usepackage[usenames,dvipsnames]{xcolor}
\usepackage{tikz}
\usepackage[T1]{fontenc}
\usepackage[utf8]{inputenc}
\usepackage{array,multirow} %
\usepackage{bm}
\usepackage{kpfonts} 
\usepackage{dsfont} 
\usepackage{setspace} 
\usepackage{cancel}
\usepackage{tipa}
\usepackage{stmaryrd}
\usepackage{csquotes}

\usepackage{hyperref,pifont}
\usepackage{todonotes}
\usepackage{dsfont}

\newcommand{\ignorer}[1]{}

\def\and{\ \wedge\ }
\usepackage[capitalize]{cleveref}
\renewcommand{\theenumi}{\roman{enumi}}
\renewcommand{\labelenumi}{\theenumi)}

\theoremstyle{plain}

\newtheorem{lemma}{Lemma}[section]

\newtheorem{theorem}[lemma]{Theorem}

\newtheorem{proposition}[lemma]{Proposition}
\newtheorem{definition}[lemma]{Definition}

\theoremstyle{definition}
\newtheorem{remark}[lemma]{Remark}

\def\eps{\varepsilon}
\renewcommand\epsilon{\varepsilon}

\def\P{\mathbb{P}}

\newcommand\E{\operatorname{\mathbb E}{}} 

\def\cF{\mathcal{F}}

\DeclareMathOperator{\Cat}{Cat}

\def\cT{\mathcal{T}}

\def\cN{\mathcal{N}}

\newcommand{\set}[1]{\left\{#1\right\}}
\newcommand\restr[2]{{%
		\left.\kern-\nulldelimiterspace %
		#1 %
		\right|_{#2} %
	}}

\author{Svante Janson}
\thanks{Supported by the Knut and Alice Wallenberg Foundation}
\address{Department of Mathematics, Uppsala University, PO Box 480,
SE-751~06 Uppsala, Sweden}
\email{svante.janson@math.uu.se}
\newcommand\urladdrx[1]{{\urladdr{\def~{{\tiny$\sim$}}#1}}}
\urladdrx{http://www.math.uu.se/~svante/}

\numberwithin{equation}{section}
\newcommand\xpar[1]{(#1)}
\newcommand\bigpar[1]{\bigl(#1\bigr)}
\newcommand\Bigpar[1]{\Bigl(#1\Bigr)}

\newcommand\bigsqpar[1]{\bigl[#1\bigr]}

\newcommand\biggsqpar[1]{\biggl[#1\biggr]}

\newcommand{\refT}[1]{Theorem~\ref{#1}}
\newcommand{\refTs}[1]{Theorems~\ref{#1}}

\newcommand{\refL}[1]{Lemma~\ref{#1}}
\newcommand{\refLs}[1]{Lemmas~\ref{#1}}
\newcommand{\refP}[1]{Proposition~\ref{#1}}

\newcommand{\refS}[1]{Section~\ref{#1}}

\newcommand{\refSS}[1]{Section~\ref{#1}}

\newcommand\hT{\widehat{T}}
\newcommand\hJ{\widehat{J}}
\newcommand\tc{d}
\newcommand\tell{\widetilde{\ell}}
\newcommand\tK{\widetilde{K}}
\newcommand\bc{\tc}
\newcommand\bK{\tK}
\newcommand\bell{\tell}
\newcommand\gs{\sigma}
\newcommand\gss{\sigma^2}
\newcommand\llrr[1]{\llbracket #1 \rrbracket}
\newcommand\ntoo{\ensuremath{{n\to\infty}}}
\newcommand\vI{\nu}

\begingroup
  \divide\count255 by 60
  \multiply\count255 by -60
  \advance\count255 by \time
  \ifnum \count255 < 10 \xdef\klockan{\the\count1.0\the\count255}
  \else\xdef\klockan{\the\count1.\the\count255}\fi
\endgroup

\newenvironment{romenumerate}[1][-10pt]{
\addtolength{\leftmargini}{#1}\begin{enumerate}
 \renewcommand{\labelenumi}{\textup{(\roman{enumi})}}%
 \renewcommand{\theenumi}{\textup{(\roman{enumi})}}%
 }{\end{enumerate}}

 \author[P. Thévenin]{Paul Thévenin}
\address{Institut für Mathematik, Oskar-Morgenstern-Platz 1, 1090 Vienna, Austria}
       \email{paul.thevenin@univie.ac.at}

\subjclass[2020]{60C05}

\begin{document}

\begin{abstract} 
We investigate here the behaviour of a large typical meandric system, proving a central limit theorem for the number of components of given shape. Our main tool is a theorem of Gao and Wormald, that allows us to deduce a central limit theorem from the asymptotics of large moments of our quantities of interest.
\end{abstract}

\maketitle

\section{Model and main result}

\subsection{Definitions and some notation}

Let $n \geq 1$ be an integer.
A \emph{meandric system} of size $n$ is a collection of non-crossing loops
in the plane that
intersect the horizontal axis exactly at the points $[2n] := \{1, \ldots,
2n \}$, we call these points the \emph{vertices} of the meandric system; two meandric systems that differ only by a continuous deformation of
the plane that fixes the horizontal axis are regarded as the same. Meandric systems were introduced, to our knowledge, by Di Francesco, Golinelli and Guitter \cite{DFGG}, and have recently become again a topic of interest \cite{Kargin, FT22, BGP}.
A meandric system can be regarded as a set of $n$ non-crossing arcs with
endpoints $[2n]$ in the upper half-plane, and another such set in the lower
half-plane; a meandric system thus determines two non-crossing matchings
(pair-partitions) of $[2n]$, one for each half-plane, and it is easily seen
that this yields a bijection between meandric systems of size $n$ and pairs
of two non-crossing matchings of $[2n]$.
In particular, since the number of non-crossing matchings of $[2n]$ is the
Catalan number
\begin{align}\label{cat}
  \Cat_n:=\frac{(2n)!}{n!\,(n+1)!},
\end{align}
see e.g.\ \cite[item 61]{Stanley},
the number of meandric systems of size $n$ is $\Cat_n^2$.
%

Each connected component of a meandric system is a single loop, intersecting
the horizontal axis in a subset of $[2n]$,
say $\set{i_1<\dots<i_{2k}}$, which we call the \emph{support}
of the loop. 
Note that necessarily there is an even number of vertices in the support, and an
even number of integers in each gap $(i_j,i_{j+1})$, i.e., $i_{j+1}-i_j$ is odd
for $1\le j<2k$.
We say that two such loops have the same shape if they differ only by a
translation. Thus we may normalize each shape to have leftmost vertex $1$,
and make the following formal definition:

\begin{definition}\label{D1}
A \emph{shape} is a (connected) non-crossing loop $S$
with support a set of integers 
$\set{ i_1=1 < i_2 < \cdots< i_{2k} = 2\ell}$, for some
$k,\ell \geq 1$, such that $i_{j+1}-i_j$ is odd for all $1 \leq j \leq 2k-1$.

Let $M$ be a meandric system and $C$ a connected component of $M$.
We say that $C$ has shape $S$ if $C$ and $S$ differ only by a translation.
%
\end{definition}

Our main theorem in the following.
We prove two special cases as \refTs{T1} and \ref{Tstrong},
and prove the remaining, more difficult, case in \refSS{SSweak}.

\begin{theorem}
\label{thm:main}
Fix a shape $S$. Let $M_n$ be a uniformly random meandric system of size $n$
(that is, on $\llbracket 1, 2n \rrbracket$) and denote by $X_{S,n}$ the number of
connected components of $M_n$ with shape $S$. Then, $X_{S,n}$ satisfies a
central limit theorem: there exist $\mu_S, \sigma_S > 0$ such that 
\begin{align}\label{tm}
\frac{X_{S,n}-n\mu_S}{\sigma_S \sqrt{n}} \underset{n \rightarrow \infty}{\overset{(d)}{\longrightarrow}} \cN(0,1),
\end{align}
where $\cN(0,1)$ denotes the standard normal distribution.
\end{theorem}

Observe that the convergence $\frac{X_{S,n}}{n} \underset{n \rightarrow
  \infty}{\overset{\P}{\rightarrow}} \mu_S$ for some constant $\mu_S$ was
already obtained in \cite{FT22}, with an explicit expression for $\mu_S$.

\section{Preliminaries}
\subsection{More notation}\label{SSnot}
For integers $m\le n$, $\llrr{m,n}$ denotes the integer interval
$[m,n]\cap\mathbb{Z}$. The \emph{size} of $\llrr{m,n}$ is its number of
points, i.e., $n-m+1$. Note that $[n]=\llrr{1,n}$.

For a component $C$ of a meandric system,
we denote by $L_C$ ($R_C$),
the leftmost (rightmost) point in the support of $C$.
Furthermore, we say that the \emph{base}
of $C$ is the interval $\llrr{L_C,R_C}$ and let $\ell(C)$ denote the
\emph{half-length} of $C$, defined as half the size of its base, i.e.,
$\ell(C):=\frac12(R_C-L_C+1)$. (Note that $\ell(C)$ always is an integer.) 
We use the same definitions for a shape $S$; then $L_S=1$ and thus
$R_C=2\ell(S)$. 

For integers $N  \ge k \geq 0$, we let 
\begin{align}
 (N)_k := N(N-1) \cdots(N-k+1)=\frac{N!}{(N-k)!}
=k!\,\binom Nk, 
\end{align}
the $k$-th descending factorial of $N$.  

We use standard $o$ and $O$ notation. Furthermore, for two (positive)
sequences $a_n$ and $b_n$, $a_n\sim b_n$ means $a_n/b_n\to1$ as \ntoo, i.e.,
$a_n=b_n(1+o(1))$, and $a_n=\Theta(b_n)$ means that there exist constants
$c>0$ and $C$ such that $c \le a_n/b_n \le C$ for sufficiently large $n$.
Note that, for example, $a_{n,r}\sim b_{n,r}$ for $r=O(\sqrt n)$ means that 
this holds for every sequence $r=r(n)=O(\sqrt n)$, which is equivalent to
$a_{n,r}\sim b_{n,r}$ uniformly for $r\le C\sqrt n$, for any $C<\infty$;
uniformity in $r$ is thus automatic in such cases.
We write ``uniformly for $r=O(\sqrt n)$'' 
for
``uniformly for $r\le C\sqrt n$, for any $C<\infty$''.
Unspecified limits are as $n\to\infty$.

\subsection{The key tool: Gao and Wormald's theorem.}

Our proof relies on a theorem due to Gao and Wormald \cite{GW04}, stating
that we can deduce a central limit theorem for a sequence of variables from
the asymptotic behaviour of their high (factorial) moments. Let us recall this
result.

\begin{theorem}[Gao \& Wormald \cite{GW04}]
\label{thm:gw}
Let $\mu_ns_n > - 1$ 
and set $\sigma_n:= \sqrt{\mu_n + \mu_n^2 s_n}$,
where $0<\mu_n \rightarrow \infty$. Suppose that $\sigma_n=o(\mu_n)$,
$\mu_n=o(\sigma_n^3)$, and that a sequence $\{X_n \}$ of nonnegative random
variables satisfies as $n \rightarrow \infty$:  
\begin{align}\label{GW}
\E[(X_n)_r] \sim \mu_n^r \exp \left( \frac{r^2 s_n}{2} \right).
\end{align}
uniformly for all integers $r$ in the range $c\mu_n/\sigma_n \leq k \leq C \mu_n/\sigma_n$, for some constants $C>c>0$. Then $(X_n-\mu_n)/\sigma_n$ converges in distribution to the standard normal as $n \rightarrow \infty$.
\end{theorem}

In other words, if  high factorial moments of a variable asymptotically
match those of a normal distribution, then convergence to the normal
distribution holds.

%

\subsection{Some lemmas}
We state some simple lemmas that will be used later.
The first is a  well known estimate that we often will  use in the sequel.
\begin{lemma}\label{Lfact}
  \begin{romenumerate}
  \item 
  If\/ $0\le k\le n/2$, then
  \begin{align}\label{lfact1}
    (n)_k = n^k\exp\Bigpar{-\frac{k^2}{2n}+O\Bigpar{\frac{k^3}{n^2}+\frac{k}{n}}}
.  \end{align}
\item 
In particular, if\/ $k=O\xpar{\sqrt n}$, then
  \begin{align}\label{lfact2}
    (n)_k = n^k\exp\Bigpar{-\frac{k^2}{2n}+o(1)}
\sim n^k\exp\Bigpar{-\frac{k^2}{2n}}
.  \end{align}
\item 
More generally, if $0\le k\le m$ with $m=O(\sqrt n)$, then
  \begin{align}\label{lfact3}
    (n-m+k)_k \sim n^k\exp\Bigpar{-\frac{m^2-(m-k)^2}{2n}}
= n^k\exp\Bigpar{-\frac{k(2m-k)}{2n}}
.  \end{align}
  \end{romenumerate}
\end{lemma}
\begin{proof}
(i), (ii):
  This follows easily from a Taylor expansion of $\log(1-i/n)$ for $0\le i<k$;
we omit the details.

(iii): This follows from (ii) and $(n-m+k)_k = (n)_m/(n)_{m-k}$.
\end{proof}

As one consequence, we obtain the following asymptotics.
\begin{lemma}\label{Lcat}
  Let $n\to\infty$ and $0\le r=O\xpar{\sqrt n}$. Then
  \begin{align}
    \frac{\Cat_{n-r}}{\Cat_n} \sim 2^{-2r}.
  \end{align}
\end{lemma}
\begin{proof}
  The definition \eqref{cat} and \refL{Lfact} yield
  \begin{align}
    \frac{\Cat_{n-r}}{\Cat_n} = \frac{(n)_r(n+1)_r}{(2n)_{2r}}
\sim \frac{(n)_r^2}{(2n)_{2r}}
=\frac{n^{2r}}{(2n)^{2r}}\exp\Bigpar{-2\frac{r^2}{2n}+\frac{(2r)^2}{4n}+o(1)}
\sim 2^{-2r}.
  \end{align}
\end{proof}

We end this section with another elementary and well known result.
\begin{lemma}
\label{lem:binomial}
Let $m, n, k \geq 1$. 
The number of unordered $k$-tuples of disjoint intervals of
size $m$ in $[n]$ is given by 
\begin{align}
\binom{n-k(m-1)}{k}.
\end{align}
\end{lemma}
\begin{proof}
  By deleting all points except the leftmost in each chosen interval,
we obtain a bijection between the set of such $k$-tuples of intervals and
the set of $k$-tuples of distinct points in $[n-k(m-1)]$.
\end{proof}

\section{A first example: components of half-length $1$.}
\label{S1}

As a warm-up, we consider first the simple case where $S$ is the loop of half-length
$1$. For any $i \in [2n]$, we let $Y_i$ be the indicator that the following
holds: 

\begin{center}
\begin{tikzpicture}[scale=1, font=\small]
\draw (0.7,0) -- (2.3,0);
\draw (1,.1) -- (1,-.1) (2,.1) -- (2,-.1) ;
\draw (2,0) arc (0:360:.5);
\draw (.8,-.2) node{$i$};
\draw (2.4,-.2) node{$i+1$};
\end{tikzpicture}
\end{center}

Then, 
\begin{align}
  \label{ab0}
X_{S,n} = \sum_{i=1}^{2n-1} Y_i
\end{align}
and thus,
for every $r\ge1$,
summing over $1\le i_1<\dots < i_r<2n$,
\begin{align}\label{ab1}
  \E[(X_{S,n})_r] &= 
\E \biggsqpar{r!\sum_{i_1<\dots<i_r} Y_{i_1}\dotsm Y_{i_r}}
=r!\sum_{i_1<\dots<i_r} \E \bigsqpar{Y_{i_1}\dotsm Y_{i_r}}.
\end{align}
The expectation in the last sum is non-zero if and only if the $r$ subintervals
$\llrr{i_j,i_j+1}$ of $\llrr{1,2n}$ are disjoint, so by \refL{lem:binomial}
there are 
$\binom{2n-r}{r}$ non-zero terms. Each of the non-zero terms is $1/\Cat_n^2$
times the number of meandric systems of size $n$ that contain $r$ given
loops of half-length $1$; by deleting these loops (and the vertices in
them), we obtain a bijection between such meandric systems and the meandric
systems of size $n-r$, and hence the number of them is $\Cat_{n-r}^2$.
Consequently, \eqref{ab1} yields
\begin{align}\label{ab2}
\E[(X_{S,n})_r] &= 
(2n-r)_r \frac{\Cat_{n-r}^2}{\Cat_n^2} 
= \frac{(2n)_{2r}}{(2n)_r}\cdot \frac{\Cat_{n-r}^2}{\Cat_n^2} 
.\end{align}
In particular, using \refLs{Lfact} and \ref{Lcat},
if $r=O\xpar{\sqrt n}$, then
\begin{align}\label{ab3}
  \E[(X_{S,n})_r] 
\sim (2n)^{2r-r}\exp\Bigpar{-\frac{4r^2}{4n}+\frac{r^2}{4n}}2^{-4r}
=
\left( \frac{n}{8} \right)^r \exp \left( -\frac{3r^2}{4n} \right)
.\end{align}
In other words,
\eqref{GW} holds (uniformly) for $0\le r\le C\sqrt n$, for any fixed
$C<\infty$,
with 
\begin{align}
  \mu_n&:=\frac{n}8,
\\
s_n&:=-\frac{3}{2n}.
\end{align}
We have $\mu_ns_n=-3/16>-1$, and thus
\begin{align}
  \sigma_n:=\sqrt{\mu_n(1+\mu_ns_n)}=\sqrt{\frac{13}{128} n}.
\end{align}
We thus have $\sigma_n=o(\mu_n)$ and $\mu_n=o(\sigma_n^3)$, and consequently
\refT{thm:gw} applies and yields:

\begin{theorem}\label{T1}
  If\/ $S$ is a simple loop of half-length 1, then
\begin{align}
\frac{X_{S,n}-n/8}{\sqrt{13 n/128}} \underset{n \rightarrow \infty}{\overset{(d)}{\longrightarrow}} \cN(0,1).
\end{align}
\end{theorem}
This is \refT{thm:main} for this particular choice of $S$, with $\mu_S=1/8$ and
$\gss_S=13/128$. 

\section{Extension to any fixed shape}\label{SS}

Let us now show how we can extend this result to any fixed shape
$S$. 
We now let $Y_i$ be the indicator that
there is a component $C$ of shape $S$ such that $L_C=i$;
note that \eqref{ab0} and \eqref{ab1} still hold.

Recall that $\ell(S)$ is the half-length of $S$, so $S$ has base
$\llrr{1,2\ell(S)}$.
We also define here three other constants $K(S), c_+(S), c_-(S)$ depending
on $S$. To avoid heavy notation, we will drop the argument $S$ in what follows,
and only denote them by $K,c_+, c_-$. 


\begin{definition}
\label{def:constants}
(See an example in Figure \ref{fig:meander_faces}.)
Observe that a component $C$ of shape $S$, taken along with the horizontal
axis, splits the plane into two unbounded faces, each belonging to one of
the half-planes, and a certain number of bounded faces. Let $F_+$ denote the
unbounded face in the upper half-plane, $F_-$ the one in the lower
half-plane, and $\cF(C)$ the set of bounded faces. 
For a face $F$, let $\vI(F)$ be the number of vertices in $\llrr{L_C,R_C}$
that lie on the boundary of $F$ but \emph{not} on $C$,
and observe that necessarily $\vI(F)$ is even. We then set
\begin{align}
K(S) &:= \prod_{F \in \cF(C)} \Cat_{\vI(F)/2},\label{ssk}\\
c_+(S)& := \vI(F_+)/2,\label{ssc+}\\
c_-(S)& := \vI(F_-)/2.\label{ssc-}
\end{align}

Note that these constants do not depend on the set of vertices on which $C$
is defined, but only on its shape $S$.
\end{definition}

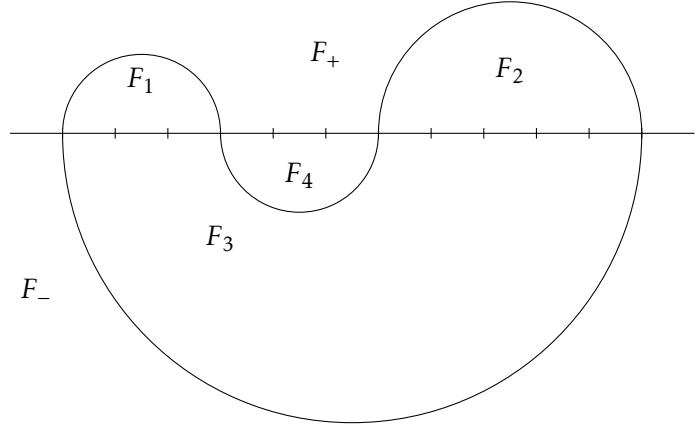
\begin{figure}[!ht]
  \[
\begin{tikzpicture}[scale=.7,font=\small]
\draw (0,0) -- (13,0);
\draw (1,.1) -- (1,-.1) (2,.1) -- (2,-.1) (3,.1) -- (3,-.1) (4,.1) -- (4,-.1) (5,.1) -- (5,-.1) (6,.1) -- (6,-.1) (7,.1) -- (7,-.1) (8,.1) -- (8,-.1) (9,.1) -- (9,-.1) (10,.1) -- (10,-.1) (11,.1) -- (11,-.1) (12,.1) -- (12,-.1);

\draw (4,0) arc (0:180:1.5);
\draw (12,0) arc(0:180:2.5);
\draw (1,0) arc (180:360:5.5);

\draw (4,0) arc (180:360:1.5);
\draw (6,1.5) node{$F_+$};
\draw (.5,-3) node{$F_-$};
\draw (4,-2) node{$F_3$};
\draw (2.5,1) node{$F_1$}; 
\draw (9.5,1.2) node{$F_2$}; 
\draw (5.5,-.8) node{$F_4$}; 
\end{tikzpicture} 
  \]
  \caption{A component $C$ with four bounded faces $F_1, F_2, F_3, F_4$.
In this example, we have $K(S)=\Cat_1^2 \Cat_2 \Cat_3=10$, $c_+(S)=1$ and
$c_-(S)=0$, where $S$ is the shape of $C$.} 
  \label{fig:meander_faces}
\end{figure}

\subsection{Strong shapes}
We say that two components \emph{overlap} if their bases overlap.
Hence, if the components have the same shape $S$, and the leftmost points in
their supports are $i$ and $j$, they overlap if $|j-i|<2\ell(S)$.

For simplicity, we study first the case when this cannot happen.
We say that a shape $S$ is \emph{strong} if two different components of a
meandric system that both have shape $S$ cannot overlap.
Thus, if $S$ is strong, then $Y_iY_j=0$ when $|j-i|<2\ell(S)$.
The simple loop in \refS{S1} and the loop in Figure \ref{fig:meander_faces} are examples of strong shapes. A shape that is not strong is called \emph{weak}; an example is
given in Figure \ref{fig:overlapping}.

\begin{proposition}
\label{prop:constants}
Let $S$ be a strong shape of half-length $\ell(S)$. Then, for all $r \geq 1$, we have
\begin{equation}
\label{eq:equality}
\E[(X_{S,n})_r] 
=  
\bigpar{2n-2r\ell(S)+r}_{r} \, 
K^r \, \frac{\Cat_{n-r\ell(S) + r c_+} \Cat_{n- r \ell(S) + r c_-}}{\Cat_n^2} 
.\end{equation}
\end{proposition}

\begin{proof}
We argue as in \refS{S1}.
As noted above, \eqref{ab1} still holds, and 
since $S$ is strong, 
we have $Y_iY_j=0$ when $|j-i|<2\ell(S)$.
Hence, the number of non-zero terms in \eqref{ab1} is
$\binom{2n-r(2\ell(S)-1)}{r}$ by \refL{lem:binomial}.
Again, all non-zero terms have the same value, which is $1/\Cat_n^2$ times
the number of ways that $r$ given disjoint loops of shape $S$ can be
completed to a meandric system of size $n$.
We can fill in the bounded faces of each component in $K$ ways,
and there are $2n-2r\ell(S)+2rc_\pm$ vertices left in the upper and lower
components, respectively, so they may be filled in in
$\Cat_{n-r\ell(S)+rc_\pm}$ ways. This yields \eqref{eq:equality}.
\end{proof}

By \refLs{Lfact} and \ref{Lcat}, it follows from \eqref{eq:equality}
that, (uniformly) for $r=O\xpar{\sqrt{n}}$, we have
\begin{align}\label{tor}
\E[(X_{S,n})_r] \underset{n \rightarrow \infty}{\sim} 
\left( \frac{2n K}{4^{2\ell(S)-c_+-c_-}} \right)^r 
\exp \left( -\frac{r^2}{4n} \left[ (2\ell(S))^2 - (2\ell(S)-1)^2 \right] \right)
.\end{align}
This is \eqref{GW} with 
\begin{align}\label{mun}
  \mu_n&:=\frac{2n K}{4^{2\ell(S)-c_+-c_-}},
\\\label{sn}
s_n&:
=-\frac{(2\ell(S))^2-(2\ell(S)-1)^2)}{2n}
=-\frac{4\ell(S)-1}{2n}.
\end{align}

In order to apply Theorem \ref{thm:gw}, we need to check that $\mu_ns_n >-1$,
which boils down to the following.
\begin{lemma}\label{LK}
  We have
\begin{equation}
\label{eq:K}
K (4\ell(S)-1) < 4^{2\ell(S)-c_+-c_-} .
\end{equation}
\end{lemma}

\begin{proof}
Observe that we can bound $K$ using the fact that $\Cat_n \leq
\frac{4^n}{n+1}$ for all $n$:
It is easy to see that for given $c_\pm$, $K$ is largest if there is only
one bounded face in each half-plane, and thus, 
\begin{align}
K \leq \Cat_{\ell(S)-c_+-1} \Cat_{\ell(S)-c_--1} \leq \frac{4^{2\ell(S)-c_+-c_--2}}{(\ell(S)-c_+)(\ell(S)-c_-)} \leq\frac{4^{2\ell(S)-c_+-c_--2}}{\ell(S)},
\end{align}
since $c_++c_- \leq \ell(S)-1$ (to see this, observe that a vertex cannot
belong to both unbounded faces of $S$, and that at least two vertices belong
to $C$).  
This yields \eqref{eq:K} directly.
\end{proof}

It is clear that $\mu_n \rightarrow \infty$.
Furthermore, we have just proved that $1+ \mu_n s_n$ is a positive constant.
Thus $\gs_n=\Theta\xpar{\sqrt{\mu_n}}$, and hence
$\sigma_n=o(\mu_n)$ and $\mu_n = o(\sigma_n^3)$.
We can therefore apply Theorem \ref{thm:gw} to obtain the central limit theorem
in this case too:

\begin{theorem}\label{Tstrong}
  Let $S$ be a strong shape.
Then
\begin{align}\label{tstrong1}
\frac{X_{S,n} - n \mu_S}{\sigma_S\sqrt{n}} \underset{n \longrightarrow \infty}{\overset{(d)}{\rightarrow}} \cN(0,1),
\end{align}
where
\begin{align}\label{tstrong2}
\mu_S = \frac{2K}{4^{2\ell(S)-c_+-c_-}} \quad\text{and}\quad 
\sigma_S = \sqrt{\frac{2K}{4^{2\ell(S)-c_+-c_-}}
  \Bigpar{1- \frac{K(4\ell(S)-1)}{4^{2\ell(S)-c_+-c_-}} }}
.\end{align}
\end{theorem}

This proves \refT{thm:main} in the case when $S$ is a strong shape, with
explicit 
formulas for $\mu_S$ and $\gs_S$.

\subsection{Weak shapes}\label{SSweak}

Finally, we study the case of a weak shape $S$.
Thus, now there may be overlaps between two components of shape $S$,
that is, two indices $i<j$ such that $|j-i|<2\ell(S)$ and $Y_i Y_j=1$,
where $Y_i$ is defined as before. See
Figure \ref{fig:overlapping} for an example. 

\begin{figure}
\begin{tikzpicture}
\draw[->] (0,0) -- (17,0);
\foreach \x in {1,...,16} \draw (\x,-.1) -- (\x,.1); 
\draw (6,0) arc (0:180:2.5);
\draw (5,0) arc(0:180:1.5);
\draw (1,0) arc (180:360:.5);
\draw (12,0) arc (0:180:2.5);
\draw (11,0) arc(0:180:1.5);
\draw (7,0) arc (180:360:.5);
\draw (6,0) arc (0:180:2.5);
\draw (5,0) arc(180:360:2.5);
\draw (10,0) arc(0:180:.5);
\draw (6,0) arc(180:360:1.5);
\draw (11,0) arc(180:360:2.5);
\draw (12,0) arc(180:360:1.5);
\draw (16,0) arc(0:180:.5);
\end{tikzpicture}
\caption{Two components of same shape overlapping. Here, $\E[Y_1 Y_7]>0$, while $2\ell(S)=10$.}
\label{fig:overlapping}
\end{figure}
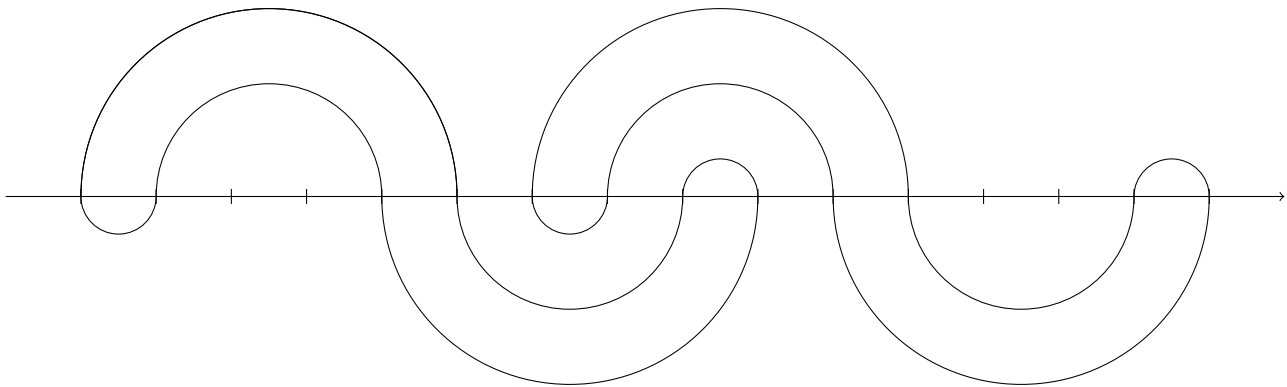

Let $A^r$ be the set of all $r$-tuples 
$ E := \{ i_1, \ldots, i_r \}$ with $1 \leq i_1< \cdots < i_r \leq 2n$, 
For any such $r$-tuple $E$,
define an equivalence relation $\sim_E$ on $\{ 1,\ldots, r \}$ as the
smallest one (for the inclusion of the equivalence 
classes) satisfying: for all $1 \leq k_1, k_2 \leq r$ such that $|i_{k_1} -
i_{k_2}| < 2\ell(S)$, $k_1 \sim_E k_2$. 
We call the equivalence classes of $\sim_E$ \emph{blocks}.
Furthermore, for $1 \leq j \leq r$, we
let $A^r_j$ be the set of $r$-tuples $E\in A^r$ 
that have exactly $j$ blocks. Thus $A^r=\bigcup_{j=1}^r A^r_j$.
Note that  $A^r_r$ is the set of $r$-tuples $E$ such that
all blocks are singletons.
An $r$-tuple $E$ corresponds to a collection
$(C_k)_1^r$ of loops of shape $S$, shifted such that $C_k$ has $L_{C_k}=i_k$.
In particular, $E\in A_r^r$ if and only if these loops are non-overlapping.

Define, for all $1 \leq u \leq r$:
\begin{align}\label{Fu}
F_u := \binom{2n-2u\ell(S)+u}{u} \, K^u \, \frac{\Cat_{n-u\ell(S) + u c_+} \Cat_{n- u \ell(S) + u c_-}}{\Cat_n^2}.
\end{align}
By the argument in the proof of \refP{prop:constants},
$u!\,F_u$ is the contribution to 
$\E[(X_{S,n})_u]$ from  $u$-tuples of non-overlapping components.

We have the  following estimates:

\begin{lemma}
\label{lem:bounds}
Let $S$ be a weak shape.
\begin{itemize}
\item[(i)] 
For all $r\ge1$,
\begin{equation}\label{lb1}
\E[(X_{S,n})_r] \geq r! F_r.
\end{equation}
\item[(ii)] For all $1 \leq u \leq r$, 
\begin{align}\label{lb2}
\sum_{E \in A^r_u} \E\left[\prod_{i \in E}Y_i\right] 
&\leq  
\binom{r-1}{u-1} (2\ell(S))^{r-u} F_u.
\end{align}

\item[(iii)]
For each fixed $M \geq 0$, uniformly for $r=O(\sqrt{n})$ with $r\ge 2M$,
\begin{align}\label{lb3}
\sum_{E \in A^r_{r-M}} \E\left[\prod_{i \in E}Y_i\right] = \Theta\left(r^M  F_{r-M}\right)
\end{align}
and, if also $r\to\infty$,
\begin{align}\label{lb12}
\sum_{E \in A^r_{r-M}(1,2)} \E\left[ \prod_{i \in E} Y_i \right] 
= \xpar{1-o(1)}\sum_{E \in A^r_{r-M}} \E\left[ \prod_{i \in E} Y_i \right] ,
\end{align}
\nobreak
where $A^r_{r-M}(1,2)$ is the subset of $A^r_{r-M}$ made only of blocks of sizes $1$ or $2$.

\end{itemize}
\end{lemma}

\begin{proof}[Proof of Lemma \ref{lem:bounds}]
(i):
We rewrite \eqref{ab1} as
\begin{align}\label{fa}
\E\left[ (X_{S,n})_r \right] 
=r! \sum_{E \in A^r} \E\left[ \prod_{i \in E} Y_i \right]
= r!\sum_{u=1}^r\sum_{E \in A^r_u} \E\left[ \prod_{i \in E} Y_i \right].
\end{align}
The term with $u=r$ yields the contribution from
$r$-tuples of non-overlapping components, which as
noted after \eqref{Fu} is $r! F_r$.

(ii):
For each $r$-tuple $E\in A^r_u$, 
keep in the product only the
leftmost point of each block,
observing that, for any sets $A \subseteq B \subseteq \llbracket 1, 2n
\rrbracket$, 
we have 
$\E\left[ \prod_{i \in B} Y_i \right] \leq \E\left[ \prod_{i  \in A} Y_i\right]$.
Note that this set of leftmost points belongs to $A^u_u$.
If the size of the $i$-th leftmost block is $j_i$, then
for each set of leftmost points, 
the number of possible positions of the other $j_i-1$ points in the block
is at most $(2\ell(S))^{j_i-1}$,
since each point after the first is within $2\ell(S)$ of the
preceding one.
Hence, 
 \begin{align}\label{lb2b}
\sum_{E \in A^r_u} \E\left[\prod_{i \in E}Y_i\right] 
&\leq  
\, \sum_{\substack{j_1+\ldots+j_u=r \\ j_1, \ldots, j_u \geq 1}} 
{\prod_{i=1}^u (2\ell(S))^{j_i-1} }\cdot 
\sum_{E' \in A^u_u} \E\left[\prod_{i \in E'}Y_i\right] 
=
\, \sum_{\substack{j_1+\ldots+j_u=r \\ j_1, \ldots, j_u \geq 1}} 
{\prod_{i=1}^u (2\ell(S))^{j_i-1} }\cdot 
F_u 
.\end{align}
Finally, 
this yields \eqref{lb2}, since
the number of allowed sequences $(j_1,\dots,j_u)$ is
$\binom{r-1}{u-1}$, and
$\prod_{i=1}^u (2\ell(S))^{j_i-1} =(2\ell(S))^{r-u}$ for all of them.

(iii): 
We partition the set $A^r_{r-M}$ as follows.
Consider an
$(r-M)$-tuple $T := (T_1, \ldots, T_{r-M})$ of integers $\geq 1$, of sum
$r$, and consider also a function $J$ which, to each $1 \leq i \leq r-M$,
associates a $T_i$-tuple $J_i$ of integers $1 =: j_{i,1} < j_{i,2} < \ldots
< j_{i,T_i}$ such that, for all $1 \leq k \leq T_i-1$, $j_{i,k+1}-j_{i,k} <
2\ell(S)$, and, furthermore, the $T_i$ loops of shape $S$ that start at the
vertices $j_k$ ($k=1,\dots,T_i$) are disjoint so that they may occur
together as components in a meandric system. 
(We call such pairs $(T,J)$ \emph{admissible}.)
Denote by $A_{T,J}$ the subset of $A^r_{r-M}$ made of
$r$-tuples $E$ such that the $i$-th leftmost block of $E$ has size $T_i$,
and if this block is $\set{a^1_i, \ldots, a^{T_i}_i}$, 
then we have $a_i^{k+1}-a_i^k = j_{i,k+1}-j_{i,k}$ for all $1 \leq k \leq T_i-1$. 
In other words, $A_{T,J}$
accounts for all $r$-tuples of components
with $r-M$ blocks, where the sizes of the blocks
are given, as well as the intervals between the starting points of each
component of shape $S$ in each block. 
Hence, $A^r_{r-M}$ is the union $\bigcup A_{T,J}$ over all admissible pairs
$(T,J)$. 

Since we only consider $(r-M)$-tuples $T$ such that
\begin{align}\label{rm}
  r=\sum_{i=1}^{r-M} T_i
= r-M+\sum_{i=1}^{r-M} (T_i-1),
\end{align}
there at most $M$ indices $i$ with $T_i>1$, and thus at least $r-2M$ indices
with $T_i=1$. Note also that if $T_i=1$, then trivially $J_i=(1)$.
Given an admissible pair $(T,J)$, we define the reduced pair
$(\hT,\hJ)$ by deleting all $T_i$ and $J_i$ such that $T_i=1$ from $T$ and $J$;
thus $\hT:=(T_i:1\le i\le r-M \text{ and }T_i>1)$ and similarly for $\hJ$.
Consequently, $\hT$ and $\hJ$ are both sequences of (the same) length $\le M$.
Since \eqref{rm} implies that their entries are bounded (for a fixed $M$),
there is only a finite
set $\cT$ of reduced pairs $(\hT,\hJ)$, where $\cT$ depends on $M$ and $S$
but not on $r$.

Conversely, given an admissible reduced pair $(\hT,\hJ)$, with
$\hT=(\hT_1,\dots,\hT_k)$, we can obtain $(\hT,\hJ)$ from $\binom{r-M}{k}$
different (admissible) pairs $(T,J)$.
Note that here, by \eqref{rm}, since each $\hT_i\ge2$,
\begin{align}
  \label{rm1}
k \le \sum_{i=1}^k(\hT_i-1)
= \sum_{i=1}^{r-M}(T_i-1)=M,
\end{align}
with equality if and only if $\hT_i=2$ for all $i\le k$.

We now want to understand the behaviour of 
$\sum_{E \in A_{T,J}} \E \left[ \prod_{i \in E} Y_i \right]$ 
for an admissible pair $(T,J)$. 
In a way similar to Proposition \ref{prop:constants}
(using an extension of \refL{lem:binomial} to intervals of different lengths), 
we obtain
\begin{align}\label{fb1}
\sum_{E \in A_{T,J}} \E \left[ \prod_{i \in E} Y_i \right]
=
\binom{2n-2\tell+(r-M)}{r-M}\, \tK\, \frac{\Cat_{n-\tc_+} \Cat_{n-\tc_-}}{\Cat_n^2},
\end{align}
where, for any  $E\in A_{T,J}$, $\tell$ 
is the sum of the half-lengths of the blocks,
$\tK$ accounts for the bounded faces defined by 
the horizontal axis and the loops defined by $E$,
and
$\tc_+, \tc_-$ for the unbounded faces.
(Note that these constants are the same for all $E\in A_{T,J}$
so they  depend only on $T$ and $J$.)
Moreover,
since at least $r-2M$ of these blocks are singletons, 
and the remaining blocks are determined by $\hT$ and $\hJ$,
we can write
\begin{align}\label{fbK}
\tK = K^{r-2M} K',   
\end{align}
for some $K'>0$ depending only on $(\hT,\hJ)$. 
Similarly,
\begin{align}\label{fbl}
  \tell &= (r-M) \ell(S) + \ell', \\
\tc_+&=(r-M)(\ell(S) - c_+) + e_+,\label{fbc+}\\ 
\tc_-& = (r-M)(\ell(S)-c_-)+e_- \label{fbc-}
\end{align}
for some $\ell', e_+,e_-$ depending only on $(\hT,\hJ)$.
In particular, for a fixed $M$, it follows that
$K', \ell', e_+, e_-$ can only take a fixed number of values independently of
$n$ and $r$. 

We compare \eqref{fb1} and $F_{r-M}$ given by \eqref{Fu}. 
First, by \refL{Lfact}(iii),
\begin{align}
& \frac{\binom{2n-2\tell + (r-M)}{r-M}}{\binom{2n-2(r-M) \ell(S) + (r-M)}{r-M}}
=
 \frac{\bigpar{2n-2\tell + (r-M)}_{r-M}}{\bigpar{2n-2(r-M) \ell(S) + (r-M)}_{r-M}}
\\\notag&\hskip4em
\sim \exp\Bigpar{-\frac{r-M}{4n}
\bigpar{(4\tell-r+M)-(4(r-M)\ell(S)-r+M)}}
=\exp\bigpar{o(1)},
\end{align}
since $\tell=r\ell(S)+O(1)$ by \eqref{fbl}
and $r=o(n)$.
Similarly, as a consequence of \refL{Lcat} and \eqref{fbc+}--\eqref{fbc-},
\begin{align}
&\frac{\Cat_{n-\tc_\pm}}{\Cat_{n-(r-M)(\ell(S)-c_\pm)}} \sim
4^{-\tc_\pm+(r-M)(\ell(S)-c_\pm)}=4^{-e_\pm}.
\end{align}
Consequently, using also \eqref{fbK}, 
we obtain from \eqref{fb1} and \eqref{Fu},
\begin{align}
\label{eq:ctj}
\frac{\sum_{E \in A_{T,J}} \E \left[ \prod_{i \in E} Y_i \right]}{F_{r-M}} =
  C_{T,J} (1+o(1)), 
\end{align}
where $C_{T,J}>0$ only depends on $(\hT,\hJ)$, 
and therefore only takes a finite number of values.
In particular,
\begin{align}\label{eq:ctj2}
\sum_{E \in A_{T,J}} \E \left[ \prod_{i \in E} Y_i \right]=\Theta\bigpar{F_{r-M}},
\end{align}
and this holds uniformly for $r=O(\sqrt n)$ and all
admissible
$(T,J)$. 

By \eqref{rm1} and the discussion before it, there are
$\binom{r-M}{k}=\Theta\xpar{r^k}$ admissible pairs $(T,J)$ for each
$(\hT,\hJ)$, where $k\le M$, with equality when all $\hT_i=2$.
Note that since we assume that the shape $S$ is weak, 
there exists at least one such admissible $(\hT,\hJ)$ with $\hT=(2,\dots,2)$.
Hence, summing \eqref{eq:ctj2} over all $(T,J)$ yields \eqref{lb3}.

Moreover, $A^r_{r-M}\setminus A^r_{r-M}(1,2)$ is the union $\bigcup'
A_{T,J}$ where we only sum over admissible pairs $(T,J)$ with some
$T_i\ge3$; these correspond to reduced pairs $(\hT,\hJ)$ with some
$\hT_i\ge3$, and we see from \eqref{rm1} that each such reduced pair
has length $\le M-1$, and thus corresponds to $O(r^{M-1})$ admissible pairs.
Consequently,
summing \eqref{eq:ctj2} over all $(T,J)$ of this type yields
\begin{align}
  \sum_{E \in A^r_{r-M}\setminus A^r_{r-M}(1,2)} \E\left[ \prod_{i \in E} Y_i \right] 
= O\bigpar{r^{M-1}F_{r-M}}
= o\bigpar{r^{M}F_{r-M}},
\end{align}
which yields \eqref{lb12}  by \eqref{lb3}.
\end{proof}

The next proposition shows that, in order to get the asymptotic behaviour
of $\E[(X_{S,n})_r]$, we only need to take into account the configurations
whose number of blocks that are not singletons is a given constant. 

\begin{proposition}
\label{prop:tightness}
Fix a weak shape $S$.  Then, there exists $\eta>0$ such that, 
for any $\epsilon >0$, there exists $M>0$ such that we have,
uniformly for $r\le\eta \sqrt{n}$,
\begin{align}\label{pt}
\sum_{u \leq r-M} \sum_{E \in A^r_u} \E\left[\prod_{i \in E} Y_{i} \right] 
\le \epsilon F_r
\leq\epsilon\frac{1}{r!} \E[(X_{S,n})_r].
\end{align}
\end{proposition}
\begin{remark}
  For convenience, we assume here that $r/\sqrt n$ is small.
In fact, \refP{prop:tightness} can easily be extended to
  $r\le C\sqrt n$ for any $C$ (with $M$ depending on $C$ and $\epsilon$),
but we have no need for this.
\end{remark}

To prove this, we start with a lemma:

\begin{lemma}
\label{lem:compare}
There exists $Q > 0$ depending only on the shape $S$ such that, for $n$
large enough, for all $u \leq  \sqrt{n}$:
\begin{align}\label{lc1}
\frac{F_{u+1}}{F_u} \geq Q \frac{n}{u}.
\end{align}
\end{lemma}

\begin{proof}
We just compute the ratio term by term, recalling \eqref{Fu}. 
We have $\frac{K^{u+1}}{K^u} = K$. The ratio of the ratios of Catalan
numbers converges uniformly to a positive constant. Finally, the ratio of
binomial coefficients is, using \refL{Lfact}, 
\begin{align}\label{lc2}
\frac{u!}{(u+1)!}\cdot\frac{\bigpar{2n-(u+1)(2\ell(S)-1)}_{u+1}} 
{\bigpar{2n-u(2\ell(S)-1)}_u}
=\frac{1}{u+1}\cdot \frac{(2n)^{u+1}}{(2n)^u}\exp\bigpar{O(1)}
\ge c\frac{n}{u}
\end{align}
for some $c>0$ and all large $n$ and $u\le \sqrt n$.
The result follows. 
\end{proof}

\begin{proof}[Proof of Proposition \ref{prop:tightness}]
Using Lemma \ref{lem:bounds}(ii), we have for all $M\geq 0$:
\begin{align}\label{fc1}
\sum_{u \leq r-M} \sum_{E \in A^r_u} \E\left[\prod_{i \in E} Y_{i} \right] 
&\le  \sum_{u=1}^{r-M} \binom{r-1}{u-1} (2\ell(S))^{r-u} F_u.
\end{align}

Letting 
\begin{align}\label{fc2}
B_{r,u} :=  \binom{r-1}{u-1} (2\ell(S))^{r-u} F_u,
\end{align}
we get from Lemma \ref{lem:compare} that, for 
$r\le \sqrt n$ and
any $u \leq r-1$:
\begin{align}\label{fc3}
\frac{B_{r,u+1}}{B_{r,u}} = \frac{1}{2\ell(S)} \frac{r-u}{u} \frac{F_{u+1}}{F_u} 
\geq \frac{Q}{2\ell(S)} \frac{n(r-u)}{u^2}
\geq \frac{Q}{2\ell(S)} \frac{n}{u^2}
.\end{align}
Hence, there exists $\eta>0$ small enough such that, for all 
$u < r \leq \eta \sqrt{n}$, we have $B_{r,u+1} \geq 2 B_{r,u}$, and thus by
backward induction, 
\begin{align}\label{fc4}
  B_{r,u} \le 2^{-(r-u)}B_{r,r}
.\end{align}
Then, for $r\le\eta\sqrt n$, \eqref{fc1} yields
\begin{align}\label{fc5}
\sum_{u \leq r-M} \sum_{E \in A^r_u} \E\left[\prod_{i \in E} Y_{i} \right] 
&\le  \sum_{u=1}^{r-M} B_{r,u}
\le 2^{1-M} B_{r,r} =2^{1-M}F_r.
\end{align}
This yields \eqref{pt} if we choose $M$ such that $2^{1-M}\le\eps$,
since  $r!\,F_r \le \E[(X_{S,n})_r]$ by the comment after \eqref{Fu}.
\end{proof}

Proposition \ref{prop:tightness} shows that we only need to understand the
asymptotic behaviour of the configurations with a number of blocks $r-M$ for
given $M \geq 0$, and Lemma \ref{lem:bounds}(iii) that we can focus on
configurations with blocks of size $1$ or $2$. To actually prove our final
result, we need to refine Lemma \ref{lem:bounds}(iii) and obtain the
explicit constants that appear.
We define another set of constants, which will account for the cases with
blocks of size 2, i.e., cases when two components of shape $S$ overlap. 

\begin{definition}
\label{def:weakconstants}
Let $S$ be a shape. There is a finite set of integers  $i \geq 1$ such that
$\E[Y_1 Y_i] > 0$ and $i-1 < 2\ell(S)$. Let $I(S)$ be this set, and $i_1,
\ldots, i_k$ its elements. For $i \in I(S)$, let $\ell_i$, $K_i$, $c_+(i)$
and $c_-(i)$ be the equivalents of $\ell(S), K, c_+, c_-$ in this case of
two components $C, C'$ that overlap and start at positions $1$ and $i$. In
particular, $\ell_i = \ell(S) + (i-1)/2$ is the total half-length of the
block made of two components of shape $S$ started at positions $1$ and
$i$. 
Furthermore, $C$ and $C'$ together with the horizontal axis
define two unbounded faces ($F_+$ in the upper half-plane  and $F_-$ in the
lower half-plane), and several bounded faces; let $\cF(C,C')$ be the
set of bounded faces. 
For each face $F $, let $\vI(F)$ be
the number of integers in $\llrr{L(C),R(C)}\cup \llrr{L(C'),R(C')}
=\llrr{L(C),R(C')}$
that are incident to $F$ but do not belong to $C$ nor to
$C'$. We set $K_i := \prod_{F \in \cF(C, C')} \Cat_{\vI(F)/2}$. Finally,
we define $c_\pm(i):=\vI(F_\pm)/2$.
Observe
again that all these constants only depend on $S$ and $i$. 
\end{definition}
Note that $i\in I(S)$ may be even; in this case $2\ell_i$, $\vI(F_+)$ and
$\vI(F_-)$ are odd, and thus $\ell_i$ and $c_\pm(i)$ are half-integers.

\begin{lemma}\label{LM}
Let $r=O(\sqrt n)$ with $r\to\infty$.
Then,
for every fixed $M\ge0$,
\begin{align}\label{mar1}
\sum_{E \in A^r_{r-M}} \E\left[ \prod_{i \in E} Y_i  \right]
\underset{n \rightarrow \infty}{\sim} 
F_r \sum_{\substack{g_i\ge0, i\in I(S)\\ \sum_i g_i=M}}
\prod_{i \in I(S)} \frac{\bigpar{b_i\frac{r^2}{2n}}^{g_i}}{g_i!}
,\end{align}
where
\begin{align}\label{mab}
b_i&:=4^{4\ell(S)-2\ell_i+c_+(i)-2c_++c_-(i)-2c_-}\frac{K_i}{K^2}.
\end{align}
\end{lemma}
Note that $b_i$ measures (in a specific way) how much two overlapping
components of shape $S$ differ from two non-overlapping ones.
\begin{proof}
For each $I(S)$-tuple $G = (g_i)_{i \in I(S)}$
of integers with sum   $M$,
let $A^r_{r-M, G}$   be the set of $r$-tuples
$1 \leq i_1 < \ldots  < i_r \leq 2n$
  with $r-2M$ blocks of size $1$ and $M$ blocks of size $2$, 
such that for each $i\in I(S)$, there are $g_i$ blocks
of type  $\set{i_k,i_{k+1}= i_k+i-1}$ with $k<r$.
Then $A^r_{r-M,G}$ is the union of some classes $A_{T,J}$ from the proof of
\refL{lem:bounds}, with all $T_i\in\set{1,2}$ and a specified number $g_i$ of 
$k$ such that $J_k=(1,i)$. 
Hence, we obtain from \eqref{fb1}, where the multinomial coefficient is the
number of $(T,J)$ that are included in $A^r_{r-M,G}$,
\begin{align}\label{ma1}
\sum_{E \in A^r_{r-M, G}} \E\left[ \prod_{i \in E} Y_i \right] 
=  \binom{r-M}{g_{i_1}, \ldots, g_{i_k}, r-2M}  \binom{2n-2\bell+(r-M)}{r-M}\, 
\bK\, \frac{\Cat_{n-\bc_+} \Cat_{n-\bc_-}}{\Cat_n^2}, 
\end{align}
where, by \eqref{fbK}--\eqref{fbc-} and the argument yielding them:
\begin{align}\label{mak}
\bK &= K^{r-2M} \prod_{i \in I(S)} K_i^{g_i},
\\\label{mal}
\bell&= (r-2M)\ell(S) + \sum_{i \in I(S)} g_i\ell_i,
\\\label{mad}
\bc_\pm &= \bell - (r-2M) c_\pm - \sum_{i \in I(S)} g_i  c_\pm(i)
.\end{align}

We now argue similarly as in the proof of \refL{lem:bounds}, but this time
we compare to $F_r$. 
We have
\begin{align}\label{ma2}
\binom{r-M}{g_1, \ldots, g_k,r-2M} \sim r^M \prod_{i \in I(S)} \frac{1}{g_i!} 
,\end{align}
\begin{align}\label{ma3}
& \frac{\binom{2n-2\tell + (r-M)}{r-M}}{\binom{2n-2r \ell(S) + r}{r}}
=
\frac{r!}{(r-M)!}\cdot
 \frac{\bigpar{2n-2\tell + (r-M)}_{r-M}}{\bigpar{2n-2r \ell(S) + r}_{r}}
\\\notag&\hskip2em
\sim r^M(2n)^{-M}\exp\Bigpar{-\frac{1}{4n}
\bigpar{(r-M)(4\tell-r+M)-r(4r\ell(S)-r)}}
\sim r^M(2n)^{-M},
\end{align}
\begin{align}\label{ma4}
\frac{\bK}{K^r} = K^{-2M} \prod_{i \in I(S)} K_i^{g_i},
\end{align}
\begin{align}\label{ma5}
\frac{\Cat_{n-\bc_\pm}}{\Cat_{n-r \ell(S) + r c_\pm}} 
\sim 4^{-\bc_\pm+r(\ell(S)-c_\pm)}
= 4^{2M(\ell(S)-c_\pm) - \sum_{i \in I(S)} (\ell_i - c_\pm(i)) g_i}.
\end{align}
and thus, from \eqref{ma1} and \eqref{Fu}, 
recalling that $\sum_{i\in I(S)}g_i=M$, 
\begin{align}\label{ma6}
&\frac{\sum_{E \in A^r_{r-M, G}} \E\left[ \prod_{i \in E} Y_i  \right]}{F_r}
    \\\notag 
  &\qquad\underset{n \rightarrow \infty}{\sim} 
r^{2M} (2nK^2)^{-M}
4^{2(\ell(S)-c_+)M - \sum_{i \in I(S)} (\ell_i - c_+(i))g_i} 
4^{2(\ell(S)-c_-)M - \sum_{i \in I(S)} (\ell_i - c_-(i))g_i}
\prod_{i \in I(S)} \frac{1}{g_i!} K_i^{g_i} 
    \\\notag 
&\qquad = 
\left( B \frac{r^2}{2n} \right)^M \prod_{i \in I(S)} \frac{q_i^{g_i}}{g_i!}
=
\prod_{i \in I(S)} \frac{\bigpar{Bq_i\frac{r^2}{2n}}^{g_i}}{g_i!}
,\end{align}
where
\begin{align}\label{ma7}
  B&:=\frac{4^{4\ell(S)-2c_+-2c_-}}{K^2},
\\\label{ma8}
q_i&:=4^{-2\ell_i+c_+(i)+c_-(i)}K_i.
\end{align}

The set $A^r_{r-M}(1,2)$ defined in \refL{lem:bounds}(iii) is
the union of $A^{r}_{r-M,G}$ over all $G$ with sum $M$.
Hence, \eqref{ma6} implies, noting that there is only a finite number of
such $G$, 
\begin{align}\label{ma9}
\sum_{E \in A^r_{r-M}(1,2)} \E\left[ \prod_{i \in E} Y_i  \right]
\underset{n \rightarrow \infty}{\sim} 
F_r
\sum_{\substack{g_i\ge0, i\in I(S)\\ \sum_i g_i=M}}
\prod_{i \in I(S)} \frac{\bigpar{Bq_i\frac{r^2}{2n}}^{g_i}}{g_i!}
.\end{align}
The result \eqref{mar1} now follows from \eqref{ma9} and \eqref{lb12}, 
using $Bq_i=b_i$.
\end{proof}

\begin{proof}[Proof of Theorem \ref{thm:main} for weak shapes]
Let  $r\to\infty$ with $r\le \eta\sqrt n$, where $\eta$ is as in
\refP{prop:tightness}. 

We may sum \eqref{mar1} over all $M\ge0$ 
(with $A^r_{r-M}:=\emptyset$ for $M>r$), since \refP{prop:tightness} shows
that we may approximate the sum by a finite sum with a fixed number of
terms.
Consequently, recalling \eqref{fa},
\begin{align}\label{mar2}
\E\left[ (X_{S,n})_r \right] 
&=r!\sum_{M=0}^\infty
\sum_{E \in A^r_{r-M}} \E\left[ \prod_{i \in E} Y_i  \right]
{\sim} 
r!\,F_r \sum_{M=0}^\infty\sum_{\substack{g_i\ge0, i\in I(S)\\ \sum_i g_i=M}}
\prod_{i \in I(S)} \frac{\bigpar{b_i\frac{r^2}{2n}}^{g_i}}{g_i!}
\\\notag&
=r!\,F_r\prod_{i \in I(S)} \exp\Bigpar{b_i\frac{r^2}{2n}}
.\end{align}
By \refLs{Lfact} and \ref{Lcat}, \eqref{Fu} implies (similarly to \eqref{tor})
\begin{align}\label{mam}
  r!\,F_r\sim 
\left( \frac{2n K}{4^{2\ell(S)-c_+-c_-}} \right)^r 
\exp\Bigpar{-\frac{r^2}{4n}\bigpar{4\ell(S)-1}}.
\end{align}
Finally,  \eqref{mar2} and \eqref{mam} yield,
for $r\to\infty$ with $r\le\eta\sqrt n$,
\begin{align}\label{mar3}
\E\left[(X_{S,n})_r\right] \underset{n \rightarrow \infty}{\sim} \left( \frac{2n K}{4^{2\ell(S)-c_+-c_-}} \right)^r \exp \left( -\frac{r^2}{4n} \left( 4\ell(S)-1 \right) +   \frac{r^2}{2n}\sum_{i \in I(S)} b_i \right).
\end{align}
This is \eqref{GW}, 
with
\begin{align}\label{mun}
  \mu_n&:=\frac{2n K}{4^{2\ell(S)-c_+-c_-}},
\\\label{sn}
s_n&:
=\frac{-(4\ell(S)-1)+2\sum_{i\in I(S)}b_i}{2n}.
\end{align}

In particular, \eqref{GW} thus holds for $r=r(n)$ with
$\frac{\eta}2\sqrt n \le r \le\eta \sqrt n$;
as noted in \refSS{SSnot}, it then automatically  holds uniformly in this range.
Furthermore, 
\begin{align}
\mu_ns_n\ge -\frac{K (4\ell(S)-1) }{4^{2\ell(S)-c_+-c_-} }>-1
\end{align}
by \refL{LK}, and we have again $\mu_n=\Theta(n)$ and $\gs_n=\Theta(\sqrt n)$.
It follows that \refT{thm:gw} applies in this case too, which yields \eqref{tm}.
\end{proof}
We obtain from \eqref{mun}--\eqref{sn}
\begin{align}
  \gss_S=\frac{2 K}{4^{2\ell(S)-c_+-c_-}}
\Bigpar{1+\frac{ K}{4^{2\ell(S)-c_+-c_-}}\Bigpar{1-4\ell(S)+2\sum_{i\in I(S)}b_i} },
\end{align}
with $b_i$ given by \eqref{mab}.
Note that this formula holds also for strong shapes (when $I(S)=\emptyset$)
by \eqref{tstrong2}.

\section{Open problems}

\renewcommand{\theenumi}{\arabic{enumi}.}
\renewcommand{\labelenumi}{\theenumi}

We list here some open problems concerning possible extensions of our results.

\begin{enumerate}
\item 
It seems possible to extend the arguments above to joint factorial moments
\begin{align}
\E\bigsqpar{(X_{S_1,n})_{r_1}\dotsm(X_{S_k,n})_{r_k}} 
\end{align}
for several shapes
$S_1,\dots,S_k$, and then obtain a multivariate version of \refT{thm:main}
using a multivariate version of Gao and Wormald's theorem
\cite{SJfringetrees}, \cite{Wormald}.
However, we have not checked the details. Such a multivariate theorem would immediately imply, for example, a central limit theorem for the number of components of a given half-length.


\item 
Considering shapes that are similar, can we obtain a central limit theorem for the number of components that only cross the
horizontal axis twice (i.e., the support has size 2, but the half-length is arbitrary)?

\item Is is true, as Kargin \cite{Kargin} has conjectured, that the total number of components is asymptotically normal?
\end{enumerate}

%
%
%

\begin{thebibliography}{99}

\bibitem{DFGG}
P. Di Francesco, O. Golinelli and E. Guitter. 
Meanders and the Temperley-Lieb algebra. 
\emph{Commun. Math. Phys.} \textbf{186} (1997), 1–-59.

\bibitem{BGP}
Jacopo Borga, Ewain Gwynne and Minjae Park.
On the geometry of uniform meandric systems.
https://arxiv.org/abs/2212.00534.

\bibitem{SJfringetrees}
Gabriel Berzunza Ojeda, Cecilia Holmgren and Svante Janson.
Fringe trees for random trees with given vertex degrees.
In preparation.

\bibitem{FT22}
Valentin Féray and Paul Thévenin.
Components in meandric systems and the infinite noodle.
\emph{Int. Math. Res. Not. IMRN} (2022), rnac156.

\bibitem{GW04}
Jason Gao and Nicholas Wormald.
Asymptotic normality determined by high moments, and submap counts of random maps.
\emph{Probab. Theory Related Fields} \textbf{130} (2004), 368--376.

\bibitem{Kargin}
Vladislav Kargin.
Cycles in random meander systems.
\emph{J. Stat. Phys.} \textbf{181} (2020), no. 6, 2322--2345.

\bibitem{Stanley}
Richard P. Stanley. \emph{Catalan numbers}. 
Cambridge University Press, New York, 2015. 

\bibitem{Wormald} 
Nick Wormald. Personal communication.

\end{thebibliography}
\end{document}